\pdfoutput=1   % make sure to use pdflatex, for the arxiv

\documentclass[a4paper, 12pt]{article}
\usepackage[a4paper, hmargin= 3cm, vmargin=2cm]{geometry}

\usepackage{graphicx}
\usepackage{amsmath, amssymb, amsfonts, amscd, amsthm, mathrsfs, mathtools, dynkin-diagrams}
\usepackage[all]{xy}
\usepackage[colorlinks, linkcolor= blue, citecolor= red]{hyperref}
\usepackage[capitalise]{cleveref}
\usepackage{pdfsync}
\usepackage{color, comment}

\usepackage[T1]{fontenc}
\usepackage[utf8]{inputenc}

\usepackage{lmodern}

\usepackage{titlesec} 

\titleformat{\section}{\centering\sc}{\S\arabic{section}.}{.5em}{}[]

\titleformat{\subsection}{\sl\bfseries}{}{0em}{}[]

\newcommand{\cop}{ \, {\textstyle \coprod} \, }

\renewcommand{\phi}{\varphi} 
\newcommand{\CC}{\mathcal{C}}

\newcommand{\z}{\mathbb{Z}}

\newcommand{\Fix}{\mathrm{Fix}}

\newcommand{\Height}{\mathrm{H}}

\newcommand{\RC}{\operatorname{RC}}

 \newcommand{\widesim}{
  \mathrel{{\scalebox{1.5}[1]{$\sim$}}}
}

\newcommand{\reducesize}[2]{%
  \mathbin{% This will be a binary math symbol (in terms of spacing around it)
    \ooalign{% Overlay a number of symbols
      \raisebox% Adjust vertical positioning of <object>
       {.4ex}% Move it down relative to current font     
          {$#1\widesim$}% First symbol (\sim in correct math style)
      \cr % Move to next symbol
      \hidewidth% Move symbol to right (~\hfill)
      \raisebox% Adjust vertical positioning of <object>
        {-.6ex}% Move it down relative to current font
        {\scalebox% Change the "font size"
          {.75}% to 50% of current font size
          {$#1#2$}% <object> in current math style
        }% \raisebox
      \hidewidth% Move symbol to left (~\hfill)
    }% \ooalign
  }% \mathbin
}% \reducesize

\newcommand{\nreducesize}[2]{%
  \mathbin{% This will be a binary math symbol (in terms of spacing around it)
    \ooalign{% Overlay a number of symbols
      \raisebox% Adjust vertical positioning of <object>
       {.4ex}% Move it down relative to current font     
          {$#1\not\widesim$}% First symbol (\sim in correct math style)
      \cr % Move to next symbol
      \hidewidth% Move symbol to right (~\hfill)
      \raisebox% Adjust vertical positioning of <object>
        {-.6ex}% Move it down relative to current font
        {\scalebox% Change the "font size"
          {.75}% to 50% of current font size
          {$#1#2$}% <object> in current math style
        }% \raisebox
      \hidewidth% Move symbol to left (~\hfill)
    }% \ooalign
  }% \mathbin
}% \reducesize

\newcommand{\stb}[1]{\mathpalette\reducesize{#1}}
\newcommand{\nstb}[1]{\mathpalette\nreducesize{#1}}

\usepackage{afterpage}

\newenvironment{verticallycentered}
               {\topskip0pt\vspace*{\fill}}
               {\vspace*{\fill}\clearpage}%
\usepackage{listingsutf8}

\lstnewenvironment{pythoncode}   % tailored for python. 
 {\lstset{
     inputencoding=utf8,
     breaklines=true,
     frame= lrtb,
     framerule= 0.5pt,
     rulecolor=\color{black!15},
    language=Python,
    backgroundcolor= \color{white},
    showstringspaces=false,
    formfeed=\newpage,
    tabsize=4,
    basicstyle=\small\ttfamily,
    keywordstyle=\color{blue}\bfseries,
    identifierstyle=,
    commentstyle=\color{gray}\itshape,
    stringstyle=\color{brown},
    morekeywords={lambda, True, False, None},
    literate=
    	{é}{{\'e}}{1}%
    	{è}{{\`e}}{1}%
    	{à}{{\`a}}{1}%
    	{â}{{\^a}}{1}%%%
    	{ç}{{\c{c}}}{1}%
    	{œ}{{\oe}}{1}%
    	{ù}{{\`u}}{1}%
    	{É}{{\'E}}{1}%
    	{È}{{\`E}}{1}%
    	{À}{{\`A}}{1}%
    	{Ç}{{\c{C}}}{1}%
    	{Œ}{{\OE}}{1}%
    	{Ê}{{\^E}}{1}%
    	{ê}{{\^e}}{1}%
    	{î}{{\^i}}{1}%
    	{ï}{{\"i}}{1}%%%
    	{ô}{{\^o}}{1}%
    	{û}{{\^u}}{1}%
}}
{} 

\lstnewenvironment{gapcode}   % modified from previous one for GAP, not great
 {\lstset{
     inputencoding=utf8,
     breaklines=true,
     frame= lrtb,
     framerule= 0.5pt,
     rulecolor=\color{black!15},
    language=Python,
    backgroundcolor= \color{white},
    showstringspaces=false,
    formfeed=\newpage,
    tabsize=4,
    basicstyle=\small\ttfamily,
    keywordstyle=\color{blue}\bfseries,
    identifierstyle=,
    commentstyle=\color{gray}\itshape,
    stringstyle=\color{brown},
    morekeywords={do, od, fi},
    literate=
    	{é}{{\'e}}{1}%
    	{è}{{\`e}}{1}%
    	{à}{{\`a}}{1}%
    	{â}{{\^a}}{1}%%%
    	{ç}{{\c{c}}}{1}%
    	{œ}{{\oe}}{1}%
    	{ù}{{\`u}}{1}%
    	{É}{{\'E}}{1}%
    	{È}{{\`E}}{1}%
    	{À}{{\`A}}{1}%
    	{Ç}{{\c{C}}}{1}%
    	{Œ}{{\OE}}{1}%
    	{Ê}{{\^E}}{1}%
    	{ê}{{\^e}}{1}%
    	{î}{{\^i}}{1}%
    	{ï}{{\"i}}{1}%%%
    	{ô}{{\^o}}{1}%
    	{û}{{\^u}}{1}%
}}
{}

\usepackage{tikz}
\usetikzlibrary{arrows}
\tikzstyle{sommet}=[circle,draw, scale=0.5]
\tikzstyle{infosommet}=[scale=0.75]
\tikzstyle{infoarete}=[midway, above, scale=0.85]

\newtheoremstyle{pedro}{}{}{\itshape}{}{\bfseries}{.}{ }{\thmname{#1}\thmnumber{ #2}\thmnote{ (#3)}}

\newtheoremstyle{pedrodef}{}{}{}{}{\bfseries}{.}{ }{\thmname{#1}\thmnumber{ #2}\thmnote{ (#3)}}

\newtheoremstyle{pedroimage}{}{}{\itshape\footnotesize}{}{\itshape\footnotesize}{.}{ }{\thmname{#1}\thmnumber{ #2}\thmnote{ (#3)}}

\theoremstyle{pedroimage}

\theoremstyle{pedro}
\newtheorem{lem}{Lemma}[section]

\newtheorem{thm}[lem]{Theorem}
\newtheorem*{thm2}{Theorem}

\newtheorem{prop}[lem]{Proposition}

\newtheorem{coro}[lem]{Corollary}

\theoremstyle{remark}

\theoremstyle{pedrodef}

\newtheorem{ex}[lem]{Example}

\title{The binary actions of alternating groups}

\author{Nick Gill \& Pierre Guillot }

\date{}

\numberwithin{equation}{section}

\begin{document}

\maketitle

\begin{abstract} 
Given a conjugacy class $\mathcal{C}$ in a group $G$ we define a new graph, $\Gamma(\mathcal{C})$, whose vertices are elements of $\mathcal{C}$; two vertices $g,h\in \mathcal{C}$ are connected in $\Gamma(\mathcal{C})$ if $[g,h]=1$ and either $gh^{-1}$ or $hg^{-1}$ is in $\mathcal{C}$.

We prove a lemma that relates the binary actions of the group $G$ to connectivity properties of $\Gamma(\mathcal{C})$. This lemma allows us to give a complete classification of all binary actions when $G=A_n$, an alternating group on $n$ letters with $n\geq 5$.

\end{abstract}

%\tableofcontents

\section{Introduction}

Let $G$ be a permutation group on a set $\Omega$. The \emph{relational complexity} of $G$ is the minimum integer $k\geq 2$ such that the orbits of $G$ on $\Omega^r$, for any $r\geq k$, can be deduced from the orbits of $G$ on $\Omega^k$. (A more precise definition will be given in \S\ref{s: background}.) Permutation groups whose relational complexity is equal to $2$ are called \emph{binary}, and the concern of this paper is to contribute to the classification of the finite binary permutation groups.

The motivation for attempting such a classification is rooted in striking results of Cherlin \cite{cherlin1} building on work of Lachlan (see, for instance, \cite{lachlan}) which show that the notion of relational complexity can be used to stratify the world of finite permutation groups in a precise sense. 

To understand how a particular finite permutation group $G$ fits into this stratification, one needs to know the relational complexity of $G$, together with another parameter, the {\em minimal number of relations} of $G$. Now Cherlin asserts that the stratification has the following property: given integers $k$ and $\ell$, the (isomorphism classes of) permutation groups of relational complexity $\le k$ and minimal number of relations $\le \ell$ fall into finitely many  infinite families, with finitely many sporadic exceptions; moreover, any permutation group, though considered sporadic in this classification for a given choice of $(k, \ell)$, will belong to one of the families for the classification corresponding to some choices $(k', \ell')$ with $k' \ge k$ and $\ell' \ge \ell$.

We stress that there is not, at present, a single pair $(k, \ell)$ for which the classification has been made explicit, although Lachlan's classification of homogeneous digraphs all but deals with the pair $(2,1)$ \cite{lachlan_digraphs}. Aside from this, though, our understanding of how this stratification works in practice is rather limited.

Over the last few years, a number of papers have been dedicated to the study of {\em binary primitive permutation groups} (i.e.\ primitive permutation groups with relational complexity equal to $2$) and a full classification of these objects is now known \cite{cherlin2, wiscons, gs_binary, ghs_binary, dgs_binary, gls_binary}.

Extending this work to cover \emph{imprimitive} binary permutation seems very difficult. A more reasonable starting point might be to understand the binary actions of important families of (abstract) groups. To this end, we propose to investigate, in a series of papers, the possibility of classifying all binary actions of groups $G$ that are \emph{almost simple}. In this paper we introduce a crucial new tool for this, and use it to deal with the alternating groups.
 
Our results all rely on the study of a graph which, so far as we are aware, is defined here for the first time: given a conjugacy class $\mathcal{C}$ in a group $G$ we define a graph, $\Gamma(\mathcal{C})$, whose vertices are elements of $\mathcal{C}$; two vertices $g,h\in \mathcal{C}$ are connected in $\Gamma(\mathcal{C})$ if $g$ and $h$ commute and either $gh^{-1}$ or $hg^{-1}$ is in $\mathcal{C}$.
 
 The connection to binary actions is achieved via the following result which is stated again, using slightly different language, as Corollary~\ref{coro-connected-comp}.
 
\begin{lem}\label{l: graph}
Let $G$ be a transitive permutation group on a set $\Omega$. Let $\CC$ be a conjugacy class of elements of prime order $p$ of maximal fixity. Let $H$ be the stabilizer of a point in $\Omega$ and let $g\in H\cap \mathcal{C}$. Then $H$ contains all vertices in the connected component of $\Gamma(\CC)$ that contains $g$.
\end{lem}
 
 Note that \emph{an element of order $p$ of maximal fixity} is simply an element of $G$ of order $p$ that fixes at least as many points of $\Omega$ as any other element of $G$ of order $p$.
 
 It turns out, for instance, that when $G=A_n$ and $p=2$, the graph $\Gamma(\mathcal{C})$ is often connected. (Proposition~\ref{prop-graphs-An} gives a precise statement.) This heavily restricts the possible transitive binary actions of $G=A_n$ for which a point-stabilizer has even order.
 
 This fact, together with a further analysis for $p$ odd, is the basis of our main result. Note that, in the following statement, we speak of a binary action, rather than a binary permutation group; this is defined in the obvious way, and allows a simpler formulation here.
 
\begin{thm}\label{thm-binary-actions-An}
Let $G$ be the alternating group $A_n$ with $n\geq 6$. Assume that there is a binary action of $G$ on the set $\Omega$. Then each orbit of $G$ on $\Omega$ is either trivial or regular.
\end{thm}

In a subsequent paper, we shall give similar results for simple groups having a single conjugacy class of involutions, among other groups. This will include $A_5$, for which the above theorem fails to hold. 

A third paper, with M. Liebeck, will fully describe the graph $\Gamma(\CC)$ when $\CC$ is a conjugacy class of involutions in a simple group of Lie type of characteristic $2$. It is our hope that $\Gamma(\CC)$ will become an object of investigation in its own right.

One naturally wonders to what extent the results for $A_n$ extend to $S_n$. It turns out that the situation is a little different for the symmetric group, as  it is possible to exhibit several families of transitive binary actions on $\Omega$, the set of cosets of a subgroup $H<S_n$. For example:
\begin{enumerate}
 \item $H=\{1\}$ or $H=A_n$;
 \item $H=\langle g\rangle$ where $g$ is an odd permutation of order $2$;
 \item $H\cong S_{n-d}$ where $1\leq d\leq n-1,$ and the action of $S_n$ on $\Omega$ is permutation isomorphic to the natural action on the set of $d$ tuples of elements from $\{1,\dots, n\}$.
\end{enumerate}
(Note that we say ``permutation isomorphic'' rather than ``permutation equivalent'' in item (3) to account for the case $n=6$.) Computation with {\tt GAP} shows that for $n=5,\dots, 10$, these are the only transitive binary actions of $S_n$.

% \subsection{Structure of the paper}

% In \S\ref{s: background} we give a full definition of a binary action, along with the more general notion of the \emph{relational complexity} of a group action. We give a number of criteria for binary actions, one of which involves the graph $\Gamma(\mathcal{C})$ -- it is essentially a stronger version of Lemma~\ref{l: graph} above. 

% In \S\ref{s: alternating} we prove Theorem~\ref{thm-binary-actions-An}, the classification of binary actions of the finite alternating groups. Our work here relies on a number of lemmas in \S\ref{s: background} which yield the classification of all binary actions of $A_n$ from the classification of all \emph{transitive} binary actions of $A_n$.

% In \S\ref{s: one-class} we prove Theorem~\ref{t: se_strong}, concerning the binary actions of simple groups that contain a single conjugacy class of involutions. Finally, \S\ref{sec-epilogue} gives the classification of binary actions for $\PSL_2(q)$.

\section{Background on relational complexity}\label{s: background}

\subsection{Complexity \& some basic criteria}

All groups mentioned in this paper are finite, and all group actions are on finite sets. We consider a group $G$ acting on a finite set $\Omega$ (on the right) and we begin by defining the {\em relational complexity}, $\RC(G, \Omega)$, of the action, which is an integer greater than $1$. 

There are at least two equivalent definitions of $\RC(G,\Omega)$. The first definition involves {\em homogeneous relational structures} on $\Omega$, and while it is useful in particular for motivation, we shall work exclusively with a different definition in terms of the action of $G$ on tuples. More information on the first definition, as well as a wealth of results on relational complexity, can be found in \cite{gls_binary}.  

So let $I, J \in \Omega^n$ be $n$-tuples of elements of $\Omega$, for some $n \ge 1$, written $I=(I_1, \ldots, I_n)$ and $J= (J_1, \ldots, J_n)$. For $r \le n$, we say that $I$ and $J$ are {\em $r$-related}, and we write $I \stb{r} J$, when for each choice of indices $1 \le k_1 < k_2 < \cdots < k_r \le n$, there exists $g \in G$ such that $I_{k_i}^g = J_{k_i}$ for all $i$. (An alternative terminology is to say that $I$ and $J$ are $r$-subtuple complete with respect to $G$.)

Now the {\em relational complexity} of the action of $G$ on $\Omega$, written $\RC(G, \Omega)$, is the smallest integer $k \ge 2$ such that whenever $n \ge k$ and $I, J \in \Omega^n$ are $k$-related, then $I$ and $J$ are $n$-related (or in other words there exists $g \in G$ with $I^g = J$). One can show that such an integer always exists and indeed, it is always less than $|\Omega|$ (or it is $2$ if $|\Omega|=1$ or $2$), see \cite{gls_binary}.  Which brings us to a quick comment about the condition $k \ge 2$: there is no universally accepted definition for actions of complexity $1$ or $0$, as several obvious putative definitions do not agree. In any case, only trivial cases would be deemed to have complexity $<2$, and at present the convention is to ignore these, so the minimum complexity that we allow is 2. When $\RC(G, \Omega)=2$, we say that the action is {\em binary}.

\begin{ex} \label{ex-natural-alternating-action}
The action of the symmetric group $S_d$ on $\{ 1,2, \ldots, d \}$ is obviously binary (a simple, general remark is that one only needs to consider triples of distinct elements, so in this example we only look at $n$-tuples with $n \le d$). On the other hand, the complexity of the action of $A_d$ on the same set is $d-1$. To see this, following Cherlin, consider the tuples $(1, 2, \ldots, d-2, d-1)$ and $(1, 2, \ldots, d-2, d)$: these are not $(d-1)$-related, but they are $k$-related for any $k < d-1$.
\end{ex}

\begin{ex}
There are two obvious examples of binary actions, for every group $G$. First, there is the case when $\Omega$ contains a single element; we call this the {\em trivial} binary action. The other canonical example is the {\em regular} binary action, that is when $\Omega = G$ with its action on itself by multiplication. It is an easy exercise to check that the latter is indeed binary; indeed one can go further and check that all {\em semiregular} actions are binary.%, and it is also a consequence of Lemma \ref{lem-ineq-RC-height} below; indeed this lemma implies that the  
\end{ex}

We are about to state a basic criterion for binariness that will be of constant use. Before we do this however, it is best to establish the next lemma, revealing the symmetry of a certain situation.

\begin{lem} \label{lem-symmetry-h1h2h3}
Let $G$ be a group. For $i=1, 2, 3$, let $H_i$ be a subgroup of $G$, and let $h_i \in H_i$. Assume that $h_1 h_2 h_3 = 1$. Then the following conditions are equivalent: \begin{enumerate}
\item there exist $h_2' \in H_2 \cap H_1$ and $h'_3 \in H_3 \cap H_1$ such that $h_1 h_2' h_3' = 1$;
\item there exist $h_1' \in H_1 \cap H_2$ and $h'_3 \in H_3 \cap H_2$ such that $h_1' h_2 h_3' = 1$;
\item there exist $h_1' \in H_1 \cap H_3$ and $h_2' \in H_2 \cap H_3$ such that $h_1' h_2' h_3= 1$.
\end{enumerate}
\end{lem}

\begin{proof}
Assume (1), so that $h_2 h_3 = h_2' h_3'$, which we rewrite as $h_3' h_3^{-1} = (h_2')^{-1} h_2 \in H_2 \cap H_3$. We have $(h_3')^{-1} (h_3' h_3^{-1})h_3 = 1$, with $(h_3')^{-1} \in H_1 \cap H_3$ and $h_3' h_3^{-1} \in H_2 \cap H_3$, so we have established (3). All the other implications are similar. 
\end{proof}

When we have elements with $h_1 h_2 h_3=1$ as above, and when the equivalent conditions of the lemma fail to hold, we think of the triple $(h_1, h_2, h_3)$ as ``minimal'' or ``optimal'', in the sense that it cannot be ``improved'' to a triple with all three elements taken from the same subgroup. As we shall see presently, the presence of such an optimal triple, with the groups $H_i$ conjugates of a single subgroup $H$, is an obstruction to the binariness of the action of $G$ on $(G:H)$.

\begin{lem}[basic criteria] \label{lem-criterion-2-not-implies-3}
    Let~$G$ act on~$\Omega $. The following conditions are equivalent:

    \begin{enumerate}
    \item There exist $I,J \in \Omega^3$ such that $I \stb{2} J$ but $I \nstb{3}J$.
    \item There are points~$\alpha_i \in \Omega$ with stabilizers~$H_i$, and elements~$h_i \in H_i$, for~$i=1,2,3$, satisfying :
  
    \begin{enumerate}
    \item $h_1 h_2 h_3 = 1$,
    \item there do NOT exist~$h_2' \in H_2 \cap H_1$, $h_3' \in H_3 \cap H_1$ with~$h_1 h_2' h_3' = 1$.
    \end{enumerate}
  
(See also Lemma \ref{lem-symmetry-h1h2h3} above.)

\item There are points~$\alpha_i \in \Omega$ with stabilizers~$H_i$, for~$i=1,2,3$, such that $H_1 \cap H_2 \cdot H_3$ is not included in $H_1 \cap (H_1 \cap H_2) \cdot H_3$.

    \end{enumerate}
    
    In particular, when these conditions hold, the action of $G$ on $\Omega$ is not binary.
  \end{lem}

  \begin{proof}
It is clear that (3) is a simple reformulation of (2). Now assume (2), and let us prove (1). Let~$\alpha_4 = \alpha_3^{h_1} = \alpha_3^{h_2^{-1}}$, using condition (a). The triples~$(\alpha_1, \alpha_2, \alpha_3)$ and~$(\alpha_1, \alpha_2, \alpha_4)$ are~$2$-related. If they were~$3$-related, we would find~$g \in H_1 \cap H_2$ taking~$\alpha_3$ to~$\alpha_4$, so~$h_1 g^{-1} \in H_3$, contrary to assumption (b). 

The converse is similar.
\end{proof}

\subsection{Removing the transitivity assumption}

Most results on relational complexity in the literature deal with transitive group actions. Nontransitive actions can be mysterious, but we can at least collect a few easy facts.

\begin{lem}
Suppose $G$ acts on $X$, and suppose that $Y \subset X$ is stable under the $G$-action. Then $\RC(G,Y) \le \RC(G, X)$.
\end{lem}

\begin{proof}
Let $k= \RC(G,X)$, and suppose that $I, J$ are $n$-tuples of elements of $Y$ which are $k$-related; then of course they are $n$-tuples of elements of $X$ which are $k$-related, so they are $n$-related. Hence the complexity on $Y$ is no greater than $k$.
\end{proof}

\begin{lem}
Let $X, Y$ be $G$-sets. Then 
\[ \RC(G, X \cop Y) = \RC(G, X \cop Y \cop Y) \, . \]

\end{lem}

Here and elsewhere we write $\coprod$ for the disjoint union.

\begin{proof}
We certainly have $\le$ by the previous lemma. Let $k= \RC(G, X \cop Y)$, and let $I, J$ be two $n$-tuples of elements of $\Omega =  X \cop Y \cop Y$ that are $k$-related. We show they are $n$-related.

For clarity, let $Y_0$ and $Y_1$ denote the two copies of $Y$ here, so $\Omega = X \cop Y_0 \cop Y_1$. The canonical bijection $Y_0 \longrightarrow Y_1$ will be written $y \mapsto y'$.

Re-arranging the indices if necessary, we may assume that $I$ is of the form $I=(x_1, \ldots, x_s, y_1', \ldots, y_t')$ for $x_i \in X \cop Y_0$, and $y_i' \in Y_1$. Since $I$ and $J$ are $1$-related, we can write $J= (z_1, \ldots, z_s, w_1', \ldots, w_t')$ with $z_i \in X \cop Y_0$ and $w_i' \in Y_1$. 

Now consider the tuples $I_0 = (x_1, \ldots, x_s, y_1, \ldots, y_t)$ and $J_0  = (z_1, \ldots, z_s$, $w_1, \ldots, w_t)$, whose entries are in $X \cop Y_0$. Inspection reveals that they are $k$-related (the bijection $y \mapsto y'$ is $G$-equivariant). So they must be $n$-related. It follows that $I$ and $J$ are $n$-related.
\end{proof}

So when studying the complexity of a non-transitive action, we only need one copy of each ``type'' of orbit. The lemma shows that repetitions do not affect the complexity.

\begin{lem}
Let $G$ act on $X \cop Y$ where the action on $X$ is free (semi-regular). Then 
\[ \RC(G, X \cop Y) = \RC(G, Y) \, . \]
\end{lem}

\begin{proof}
From the previous lemma, or alternatively by an obvious induction, we may as well assume that the action on $X$ is regular, so it is in particular binary. 

We only need to show $\le$. Let $I, I'$ be $n$-tuples of elements of $X \cop Y$, say $I=(x_1, \ldots, x_s, y_1, \ldots y_t)$ and $I'=(x'_1, \ldots, x'_s, y'_1, \ldots y'_t)$ in obvious notation. Assume that $I, I'$ are $k$-related, where $k \ge 2$ is the complexity of $Y$. There is nothing to prove if $s=0$.

As the action on $X$ is binary, we may as well assume that $x_i' = x_i$ for all $i$. Further, suppose $g \in G$ is such that $(x_1, y_j)^g = (x_1, y_j')$. As the action is free, we must have $g=1$ and $y_i' = y_i$, so $I=I'$.
\end{proof}

Here is a similar statement. The easy proof will be omitted.

\begin{lem}
Let $G$ act on $X \cop Y$ where the action on $X$ is trivial. Then 
 \[ \RC(G, X \cop Y) = \RC(G, Y) \, . \]
\end{lem}

We shall see later that the following corollary applies in particular to the alternating groups.

\begin{coro} \label{coro-nontransitive}
Let $G$ be a group whose only transitive, binary actions are the regular one and the trivial one. Then any binary action of $G$ is a disjoint union of copies of the regular action, and copies of the trivial action.
\end{coro}

\subsection{Another criterion for binary actions}

We start with a lemma adapted from \cite{ghs_binary}. In the lemma $G$ acts on $\Omega$ with $\Lambda \subset \Omega$. We write $G_\Lambda$ for the setwise stabilizer of $\Lambda$, $G_{(\Lambda)}$ for the pointwise stabilizer of $\Lambda$ and $G^\Lambda=G_\Lambda/G_{(\Lambda)}$ for the permutation group induced by $G$ on $\Lambda$.

\begin{lem}
Suppose there is a subset $\Lambda \subset \Omega$ with more than 2 elements, and with the following properties: there is a permutation $\tau $ of $\Lambda$ with $\tau\not\in G^\Lambda$ and there are permutations $\eta_1, \eta_2, \ldots, \eta_r$ of $\Lambda$ such that:
\begin{enumerate}
\item $g_i := \tau \eta_i \in G^\Lambda$,
\item the support of $\tau$ and that of $\eta_i$ are disjoint, for all $i$, 
\item any $\lambda \in \Lambda$ is fixed by at least one $\eta_i$. 

\end{enumerate}
Then the action of $G$ on $\Omega$ is not binary.
\end{lem}

\begin{proof}
Let $\lambda_1, \ldots, \lambda_t$ be the elements of $\Lambda$, and consider the tuples $I= (\lambda_1, \ldots, \lambda_t)$ and $J=(\lambda_1^\tau, \ldots, \lambda_t^\tau)$. The assumption that $\tau\not\in G^\Lambda$ means that there is no element of $G$ taking $I$ to $J$, and so it is enough to show that these are $2$-related.

Without loss of generality, let us look at the first two entries in the tuples. There are 4 cases. If $\lambda_1^\tau = \lambda_1$ and also $\lambda_2^\tau = \lambda_2$, the identity takes $(\lambda_1, \lambda_2)$ to $(\lambda_1^\tau, \lambda_2^\tau )$. If $\lambda_1^\tau = \lambda_1$ but $\lambda_2^\tau \ne \lambda_2$, then $\lambda_2$ and $\lambda_2^\tau$ are in the support of $\tau$, so they are not in the support of $\eta_j$ for all $j$. Now choose $\eta_i$ fixing $\lambda_1$. Thus $\eta_i$ takes $(\lambda_1^\tau, \lambda_2^\tau)$ to itself, and the element $g_i=\tau \eta_i \in G^\Lambda$ takes $(\lambda_1, \lambda_2)$ to $(\lambda_1^\tau, \lambda_2^\tau)$. There is another symmetric case.

Finally, suppose $\lambda_1^\tau \ne \lambda_1$ and $\lambda_2^\tau \ne \lambda_2$. All four points mentioned are thus in the support of $\tau$, so not in the support of any $\eta_i$, or in other words they are fixed by any $\eta_i$. Again $g_i = \tau \eta_i$ does the job.
\end{proof}

In the coming example, and in the rest of the paper, we use the notion $\Fix(g)$ for the set of fixed points of $g \in G$ in some action which is implicit from the context.

\begin{ex} \label{e: inter}
Again let $G$ act on $\Omega$. Suppose $g$ and $h$ are two elements of $G$ which commute, and let 
\[ \Lambda = \Fix(g) \cup \Fix(h) \cup \Fix(gh^{-1}) \, . \]
(The union might not be disjoint.) Then $\Lambda$ is stable under the actions of $g$ and $h$, because these two commute.

By construction $g$ and $h$ induce the same permutation $\tau$ on $\Fix(gh^{-1})$ ; extend it to be the identity outside of $\Fix(gh^{-1})$, so we can look at it as a permutation of $\Lambda$. Next, let $\eta_1$ be the permutation induced by $g$ on $\Fix(h)$, again viewed as a permutation of $\Lambda$, and symmetrically let $\eta_2$ be the permutation induced by $h$ on $\Fix(g)$, viewed as a permutation of $\Lambda$.

Let us check whether the conditions of the lemma are met. First, the supports of $\tau$ and $\eta_i$ are certainly disjoint: by definition $\eta_1$ only moves points that are fixed by $h$ and so also by $\tau$ if these points are in $\Fix(gh^{-1})$, which is where $\tau$ could be nontrivial. Similarly for $\eta_2$.

Next, the permutation induced on $\Lambda$ by $g$ is $\tau \eta_1$, so $\tau \eta_1 \in G^\Lambda$, and similarly for $\tau \eta_2$. We have conditions (1) and (2).

As for condition (3), take $\lambda \in \Lambda$. If $\lambda$ lies neither in $\Fix(g)$ nor $\Fix(h)$, then it is fixed by both $\eta_1$ and $\eta_2$, by their definition. If $\lambda \in \Fix(g)$, then either $\lambda \in \Fix(gh^{-1})$ (in which case it is fixed by $h$ as well, so by everything), or $\lambda \not\in \Fix(gh^{-1})$, so it is fixed by $\tau$ by definition, and so, being fixed by $h = \tau \eta_2$, it must be fixed by $\eta_2$.

We have the three conditions of the lemma. The conclusion is that {\em if we can only prove that $\tau \not\in G^\Lambda$, and that $\Lambda$ has more than two elements, then the action will not be binary}. 
\end{ex}

The integer $|\Fix(g)|$ will be called the {\em fixity} of $g \in G$. For a prime $p$ we say that $g$, an element of order $p$, has {\em maximal $p$-fixity} when no element of $G$ of order $p$ fixes more points than $g$ in the action under scrutiny. 
 Also, the next statement uses the notation $\Fix(K)$ for the set of common fixed points of all the elements of the subgroup $K$ of $G$.

\begin{lem} \label{l: inter}
Suppose that $G$ acts on $\Omega$, that $p$ is a prime, and that $g$ and $h$ are commuting $p$-elements of $G$ with the property that $\langle g\rangle$, $\langle h\rangle$ and $\langle gh^{-1}\rangle$ are conjugate subgroups of $G$. We write $K=\langle g, h\rangle$. Suppose that
\begin{enumerate}
 \item $|\Fix(K)|<|\Fix(g)|$;
 \item $g$ has maximal $p$-fixity.
\end{enumerate}
Then the action of $G$ on $\Omega$ is not binary.
\end{lem}

\noindent Note that, in this situation, any nontrivial power of either $g$, $h$ or $gh^{-1}$ has maximal $p$-fixity. Note too that the first supposition implies that $g\neq h$.

\begin{proof}
As in Example \ref{e: inter}, we introduce $A = \Fix(g)$, $B=\Fix(h)$ and $C = \Fix(gh^{-1})$, and we put
\[ \Lambda = A \cup B \cup C \, . \]
Here we have $A \cap B = A \cap C = B \cap C = A \cap B \cap C  = \Fix(K)$.

We use the fact that $\langle g\rangle$, $\langle h\rangle$ and $\langle gh^{-1}\rangle$ are conjugate subgroups of $G$. Thus, there is an integer $r$ such that 
\[
 |\Fix(g)|=|\Fix(h)|=|\Fix(gh^{-1})|=r.
\]
Put $r'= |\Fix(K)|$, so $r' < r$ by hypothesis. We deduce from the calculations above that the size of $\Lambda$ is $3r - 3r' + r' = 3(r - r') + r' > 2$.

Now let $\tau, \eta_1$ and $\eta_2$ be as in Example \ref{e: inter}, and let us prove that the permutation $\tau$ is not induced by an element $x \in G$ ; this will suffice. Assume for a contradiction that $x$ exists. As $\tau$ is not the identity ($\Fix(gh^{-1})$ being nonempty and not equal to $\Fix(g)$), we see that the order of $x$ is divisible by $p$. Thus some power of $x$, say $s \in G$, is a $p$-element and fixes at least as many points as $x$. By definition, $\tau$ is the identity on $\Lambda - C$, and it is also the identity on $\Fix(K)$, so in the end we see that $\tau$ fixes all the elements of $A \cup B$, as does $x$. However, the size of $A \cup B$ is $r + (r - r') > r$, which means that $s$ is a $p$-element of $G$ that fixes more elements of $\Omega$ than $g$. This is a contradiction and we are done.
\end{proof}

\begin{coro} \label{c: inter}
Suppose that $G$ acts on the set of cosets of a subgroup $H$, that $p$ is a prime, and that there exists $K$ an elementary-abelian subgroup of $G$ of order $p^2$ satisfying the following properties:
\begin{enumerate}
 \item $K=\langle g, h\rangle$ for some $p$-elements of maximal $p$-fixity, $g,h\in G$;
 \item $\langle g\rangle$, $\langle h\rangle$ and $\langle gh^{-1}\rangle$ are conjugate subgroups of $G$;
 \item $K\cap H=\langle g\rangle$.
 \end{enumerate}
 Then the action of $G$ on the cosets of $H$ is not binary.
\end{coro}

\begin{proof}
This is just a reformulation of the lemma, with $\Omega = (G:H)$. Indeed, if $\alpha$ denotes $H$ as an element of $(G:H)$, then the stabilizer of $\alpha$ is $H$. Thus, we see that $\alpha \not \in \Fix(K)$ but, since $K\cap H=\langle g\rangle$, we certainly have $\alpha\in \Fix(g)$. As $\Fix(g) \ne \Fix(K)$, the lemma applies.
\end{proof}

% Note that, in the special case where $p=2$, we can replace the third condition with $|K\cap H|=2$.

\subsection{The graphs on conjugacy classes}

Let us recast the above criteria in terms of certain graphs and their connected components. Recall that the {\em rational conjugacy class} of $g \in G$ is the set of $h \in G$ such that $\langle g \rangle$ and $\langle h \rangle$ are conjugate subgroups of $G$. If $\CC^{rat}$ is a rational conjugacy class whose elements have order $p$ (a prime), we define the graph $\Gamma(\CC^{rat})$ whose set of vertices is $\CC^{rat}$, and with an edge between $x$ and $y$ when they commute and $xy^{-1} \in \CC^{rat}$. 

% Let $G$ be a group and let $\CC$ be a conjugacy class of $p$-elements in $G$, where $p$ is a prime. 

% \begin{enumerate}
%  \item vertices are elements of $\CC$;
%  \item we join $x,y\in \CC$ with an edge provided (i) $x$ and $y$ commute and (ii) $xy^{-1}\in\CC$. 
% \end{enumerate}

Note that if $p=2$, a rational conjugacy class is just a conjugacy class, and moreover, two elements $x,y\in \CC^{rat}$ are joined by an edge in $\Gamma(\CC^{rat})$ if and only if $xy\in\CC^{rat}$.

Let us see why this graph is useful in the study of binary actions. 

\begin{lem}\label{l: graph 2}
 Let $G$ act on the set of cosets of a subgroup $H$ and assume that $p$ is a prime dividing $|H|$. Let $\CC^{rat}$ be a rational conjugacy class of $p$-elements of maximal $p$-fixity and let $X=\CC^{rat}\cap H$. If the action of $G$ is binary, then $X$ and $\CC^{rat}\setminus X$ are not connected to each other in the graph $\Gamma(\CC^{rat})$.
\end{lem}
\begin{proof}
 We prove the contrapositive: suppose $g\in X$, $h\in\CC^{rat}\setminus X$ and $g$ and $h$ are connected by an edge in $\Gamma(\CC^{rat})$. Then $K=\langle g,h\rangle$ is an elementary-abelian $p$-subgroup of $G$ satisfying all of the suppositions of Corollary~\ref{c: inter}. We conclude that the action of $G$ is not binary.
\end{proof}

With the next formulation, we complete the switch to graph-theoretical language. The proposition is quite general, and the rest of the paper relies only on a special case described below.

\begin{prop} \label{prop-connected-comp}
Let $G$ act on the set of cosets of a subgroup $H$, and assume that the action is binary. Let $p$ be a prime dividing $|H|$, and let $\CC^{rat}$ be a rational conjugacy class of $p$-elements of $G$ of maximal $p$-fixity. Finally, suppose that $\Gamma$ is any subgraph of $\Gamma(\CC^{rat})$.

Then for any $g \in  H$ which is a vertex of $\Gamma$, the connected component of $\Gamma$ containing $g$ lies entirely in $H$. 
\end{prop}

\begin{proof}
% %Pick $g \in \CC \cap H$, which exists by assumption. 
% By the last lemma, all the vertices in the connected component of $\Gamma(\CC)$ containing $g$ must belong to $H$.
Immediate from the last lemma.
\end{proof}

From now on, we shall always work with the following subgraph $\Gamma$, which is reasonably easy to describe in concrete cases. Let $\CC$ be a conjugacy class of $p$-elements in $G$; it is contained in a unique rational conjugacy class $\CC^{rat}$. We define $\Gamma(\CC)$ to be the graph whose vertices are the elements of $\CC$, and with an edge between $x, y \in \CC$ if and only if \begin{enumerate}
\item $x$ and $y$ commute, 
\item either $xy^{-1} \in \CC$ or $yx^{-1} \in \CC$.
\end{enumerate}
In this way $\Gamma(\CC)$ is a subgraph of $\Gamma(\CC^{rat})$, and the proposition applies. (For $p=2$ there is no difference between $\Gamma(\CC)$ and $\Gamma(\CC^{rat})$.)

%\begin{ex}
%While studying the group $\PSU_3(q)$, we shall encounter an example where $\Gamma(\CC)$ is isomorphic to the set of nonzero cubes in $\f_q$ (where $q$ is even in this instance), with an edge between $x$ and $y$ if and only if $x+y$ is also a cube: see Lemma \ref{lem-cubes}. A little later, as we turn our attention to $\PSL_2(q)$, we shall be in a situation where a certain subgraph of $\Gamma(\CC)$ is similarly defined with squares instead of cubes: see Lemma \ref{l: p elements}. It is easy to produce more examples of this kind, close cousins of the familiar Payley graphs, by considering a semidirect product $G= \f_q \rtimes \f_q^\times$, with the action of $\f_q^\times$ of your choice.
%\end{ex}

Some vocabulary in order to state a corollary. When $g \in \CC$, the {\em component group of $g$ in $\Gamma(\CC)$} is the subgroup of $G$ generated by all the elements in the connected component of $\Gamma(\CC)$ containing $g$. The {\em component groups of $\Gamma(\CC)$} are the various groups thus obtained by varying $g$ ; they are all conjugate in $G$. The next corollary asserts that, in the case of transitive, binary actions, each stabilizer must contain a component group. 

\begin{coro} \label{coro-connected-comp}
Let $G$ act on the set of cosets of a subgroup $H$, and assume that the action is binary. Let $p$ be a prime dividing $|H|$, and let $\CC$ be a conjugacy class of $p$-elements of $G$ of maximal $p$-fixity. Then for any $g \in \CC \cap H$, the component group of $g$ in $\Gamma(\CC)$ is contained in $H$.
\end{coro}

If particular, suppose that $\Gamma(\CC)$ is connected and that $G$ is simple. Then $H$ must contain the subgroup generated by $\CC$, which is normal, and we conclude in this situation that $H=G$.

%\begin{ex}\label{ex: edgeless}
%While most uses of the corollary in this paper will exploit the fact that $\Gamma(\CC)$ has sometimes large connected components, thus preventing the stabilizers in certain binary actions from being too small, and thereby giving an obstruction to the very existence of nontrivial binary actions, here is an easy situation where we can do the opposite.  

%Suppose that $\CC$ is a conjugacy class of involutions in a group $G$, let $h\in \CC$ and $H=\langle h\rangle$, and suppose that $\Gamma(\CC)$, the graph described above, does not have any edges. This means that if $H_1, H_2$ and $H_3$ are conjugates of $H$, then $H_1\cap H_2.H_3=\{1\}$ and Lemma~\ref{lem-RC-height-2} then implies that the action of $G$ on the set of right cosets of $H$ in $G$ is binary.

%To see that such actions exist, let $G=\PSL_n(q)$ with $nq$ odd and take $\mathcal{C}$ to be the conjugacy class of involutions that are the projective image of a matrix in $\SL_n(q)$ with a $(-1)$-eigenspace of dimension $n-1$ and a $1$-eigenspace of dimension $1$. It is easy to see that if $g,h\in \CC$ then $gh$ has a $1$-eigenspace of dimension at least $n-2$. 

%Thus, for $n\geq 5$, the graph $\Gamma(\CC)$ is edgeless and the group $\PSL_n(q)$ has a transitive binary action with a point-stabilizer of order $2$.

% The problem of classifying all simple groups $G$ with a conjugacy class $\CC$ of involutions for which $\Gamma(\CC)$ is edgeless is a tantalising one!
%\end{ex}

\section{Alternating groups} \label{s: alternating}

In this section, we prove Theorem \ref{thm-binary-actions-An}, which was given in the introduction, and whose statement we reproduce here for convenience.

\begin{thm2}
Let $G$ be the alternating group $A_n$, for $n \ge 6$. Assume that there is a binary action of $G$ on the set $\Omega$. Then each orbit of $G$ on $\Omega$ is either trivial or regular.
\end{thm2}

Corollary \ref{coro-nontransitive} allows us to assume that the action of $G$ is transitive. So we consider the action of $G$ on the cosets of a subgroup $H$; we suppose that $H \ne 1$ and our aim is to prove that $H=G$.

The first step of the proof will show that $H=G$ if we know that the order of $H$ is even -- this step will be completed at the end of the next subsection. After this, we shall show that assuming $H$ to have odd order leads to a contradiction.

\subsection{Involutions}
Here we collect facts about the graphs $\Gamma(\CC)$ when $\CC$ is a conjugacy class of involutions in $A_n$, reaching a complete description of the connected components. The first step of the proof of Theorem \ref{thm-binary-actions-An} will then be trivial to complete, thanks to Corollary~\ref{coro-connected-comp}.

Some definitions. A {\em quad} is a permutation of the form $(ab)(cd)$, with $a, b, c, d$ distinct. It will be useful, for what follows, to keep in mind that there are just three quads with support $\{ a, b, c, d \}$, namely $(ab)(cd)$, $(ac)(bd)$ and $(ad)(bc)$ ; they are the nonidentity elements in a Klein group which acts regularly on $\{ a, b, c, d \}$. 

% We say that there is a {\em relation of type 1 between the quads $q$ and $q'$} when it is possible to write $q=(ab)(cd)$ and $q'= (ac)(bd)$. 

When $q=(ab)(cd)$ and $q' = (ab)(c'd')$, with $\{ c, d \} \cap \{ c', d' \} = \emptyset$, we say that the quads $q$ and $q'$ are {\em related by the exchange $(cd) \leftrightarrow (c'd')$}.

% Let $\sigma$ be any permutation. We say that there is a {\em relation of type 2 between the quads $q$ and $q'$ with respect to $\sigma$} when $q= (ab)(cd)$ and $q'= (ab)(xy)$ where $x, y$ are fixed by $\sigma$. (In all examples, $x$ and $y$ will be fixed by $q$ as well, so $x, y \not\in \{ a, b, c, d \}$.) We will also say that $x$ and $y$ are the {\em fixed points of the relation of type 2}.

\begin{lem} \label{lem-edge-quads}
Let $\CC$ be a conjugacy class of involutions in $S_n$. Let $s, t \in \CC$. Then there is an edge in $\Gamma(\CC)$ between $s$ and $t$ if and only if there is an integer $k$, quads $q_1, \ldots, q_k$ with disjoint supports satisfying $s = q_1 \cdots q_k$, and quads $q_1', \ldots, q_k'$ with disjoint support satisfying $t= q_1' \cdots q_k'$, such that for every $i$ we have either : \begin{enumerate}
\item $q_i$ and $q_i'$ have the same support and are distinct, or 
\item $q_i$ and $q_i'$ are related by the exchange $(x_iy_i) \leftrightarrow (u_iv_i)$, where $x_i$ and $y_i$ (resp.\ $u_i$ and $v_i$) are fixed by $t$ (resp.\ by $s$), and all the elements $x_i, y_i, u_i, v_i$ thus introduced as $i$ varies are distinct.
\end{enumerate}

% either a relation of type 1 between $q_i$ and $q_i'$, or a relation of type 2 with respect to $s$, with the property that the fixed points of the relations of type 2 are all distinct as $i$ varies.

In particular, if the size of the support of $s$ (or any other element of $\CC$) is not divisible by $4$, then there is no edge in $\Gamma(\CC)$ at all.
\end{lem}

\begin{proof}
Suppose there is an edge between $s$ and $t$. The group $K = \langle s, t \rangle$ is a Klein group, and we may write the disjoint union
\[ \{ 1, 2, \ldots, n \}  = A \cup B \cup C \cup D \cup E \]
according to the types of orbits of $K$. More precisely, let $A$ be the set of elements with trivial stabilizer, so that the action of $K$ on $A$ is free (semiregular). Each orbit of $K$ on $A$ is of the form $\{ a, b, c, d \}$ where $s$ acts as $(ab)(cd)$ and $t$ acts as $(ac)(bd)$. In other words, the restriction of $s$ to $A$ is given by a product $q_1 \cdots q_a$ of quads with disjoint supports, while $t$ is given by $q_1' \cdots q_a'$, and $q_i$ has the same support as $q_i'$. 

% Note that the permutations of $A$ induced by $s$, $t$ or $st$ are all conjugate within the symmetric group of $A$. \pierre{do we need this sentence at all?}

The set $E$ will be that of fixed points of $K$. Now let $B$ resp.\ $C$ resp.\ $D$ be the set of elements whose stabilizer is $\langle s \rangle$ resp.\ $\langle t \rangle$ resp.\ $\langle st \rangle$. The action of $t$ on $B$ is given by a product $\beta_1 \cdots \beta_b$ where each $\beta_i$ is a transposition (that is $\beta_i = (xy)$ for some $x, y \in B$), while $s$ acts trivially on $B$. Similarly, the action of $s$ on $C$ is given by $\gamma_1 \cdots \gamma_c$ while $t$ acts trivially on $C$, and finally $st$ acts trivially on $D$, on which $s$ and $t$ both act as $\delta_1 \cdots \delta_d$. Restricting to $B \cup C \cup D$, we may write 
\[ s = \gamma_1 \cdots \gamma_c \delta_1 \cdots \delta_d \, , \quad t = \beta_1 \cdots \beta_b \delta_1 \cdots \delta_d \, , \quad
st = \gamma_1 \cdots \gamma_c \beta_1 \cdots \beta_b \, . \]

Using that $s$, $t$ and $st$ are conjugate to one another, we conclude that $b= c= d= \ell$, say. Consider now the quad $\delta_i \gamma_i$ for $1 \le i \le \ell$, as well as the quad $\delta_i \beta_i$. If $\gamma_i = (x_i y_i)$ and $\beta_i = (u_i v_i)$, then the quads are related by the exchange $(x_i y_i) \leftrightarrow (u_i v_i)$. This proves the existence of the quads as announced. 

For the converse, one simply works backwards.
\end{proof}

When a conjugacy class $\CC$ of $S_n$ has edges in its graph, it must therefore be contained in $A_n$ rather than just $S_n$. Moreover, $\CC$ then still forms a conjugacy class of $A_n$.

\begin{ex} \label{ex-compute-quads}
Let $\CC$ denote the conjugacy class of a single quad in $A_n$ (or $S_n$, this is the same). Assuming that $n \ge 6$, we show that $\Gamma(\CC)$ is connected. Let us write $s \equiv t$ when $s$ and $t$ are connected by an edge. We can begin by computing : 
\begin{align*}
(12)(34)  & \equiv (14)(23) & \textnormal{(same support)} \\
          & \equiv (14)(56) & \textnormal{(exchange of fixed points)} \\
          & \equiv (16)(45) & \textnormal{(same support)} \\
          & \equiv (23)(45) & \textnormal{(exchange of fixed points)}
\end{align*}

Swapping $1$ and $2$ in this computation (or in other words, conjugating everything by $(12)$), we get a path between $(21)(34) = (12)(34)$ and $(13)(45)$. Swapping $1$ and $3$ gives a path between $(32)(14) \equiv (12)(34)$ and $(21)(45)$, and finally, swapping $1$ and $4$ produces a path between $(42)(31) \equiv (12)(34)$ and $(23)(15)$. 

We have shown that, for every subset $\Lambda$ of $\{1,2,3,4,5\}$ of cardinality $4$, there is a path connecting $(12)(34)$ to a quad $q$ with support equal to $\Lambda$. Since all quads with the same support are connected to each other, we conclude that if $q$ is a quad with support contained in $\{ 1, 2, 3, 4, 5 \}$, then there is a path from $(12)(34)$ to $q$.

Now more generally, suppose the support of $q$ is $\{ a, b, c, d\}$ with $a<b<c<d$, and that $d \ge 6$. Then there are two elements $x$ and $y$ with $1 \le x < y < d$ which are fixed by $q$. We have $q \equiv (ab)(cd) \equiv (ab)(xy)$, and the maximal element in the support of $(ab)(xy)$ is less than $d$. Iterating, we have a path between $q$ and a quad whose support is in $\{ 1, 2, 3, 4, 5 \}$ as above. We have shown the existence of a path between $(12)(34)$ and any quad.

We note that this result does not hold when $n=5$. In this case the graph $\Gamma(\CC)$ is made up of five disjoint triangles, each corresponding to three quads with a shared fixed point. For instance, one of these triangles has vertices $(12)(34)$, $(13)(24)$ and $(14)(23)$, the three quads in $A_5$ that fix the point $5$.
\end{ex}

The next lemma provides a statement in which only one quad is modified at a time (note that, in the notation of Lemma \ref{lem-edge-quads}, the quads $q_i$ and $q_i'$ are distinct for all $i$). The arguments in the previous example will then generalize almost immediately.

\begin{lem}
Let $\CC$ be a conjugacy class in $A_n$, and let 
\[ s = q_1 \cdots q_k \textit{ and } t = q_1' q_2 \cdots q_k \]
be elements of $\CC$. Assume that $q_1,\dots, q_k$ (resp. $q_1',q_2,\dots, q_k$) are quads with disjoint supports. Assume, moreover, that $q_1'$ is any quad such that either $q_1$ and $q_1'$ have the same support, or $q_1$ and $q_1'$ are related by an exchange $(xy) \leftrightarrow (uv)$, where $x$ and $y$ are fixed by $t$, and $u$ and $v$ are fixed by $s$.
   
Then $s$ and $t$ belong to the same connected component of $\Gamma(\CC)$.
\end{lem}

\begin{proof}
Of course there is nothing to prove if $q_1' = q_1$, so assume $q_1' \ne q_1$. Let $s_0 = s$ and $s_3 = t$. We shall introduce $s_1, s_2 \in \CC$ such that there is an edge between $s_j$ and $s_{j+1}$ for $j=0, 1, 2$.

Let $i > 1$ and let $q_i = (ab)(cd)$. Put $q_i' = (ac)(bd)$, $q''_i = (ad)(bc)$. In this way, the quads $q_i, q_i'$ and $q_i''$ are the three distinct quads with support $\{ a, b, c, d \}$.

When $i=1$, we have already $q_1$ and $q_1'$ ; if $q_1' = (ab)(cd)$, we introduce $q_1''= (ad)(bc)$, so that $q_i'$ and $q_i''$ are distinct, of the same support.

Now let 
\[ s_1 = q_1' q_2' \cdots q_k' \qquad \textnormal{and}\qquad s_2 = q_1'' q_2'' \cdots q_k'' \, . \]
Applying lemma \ref{lem-edge-quads} several times shows the existence of an edge between $s_j$ and $s_{j+1}$, for $j= 0, 1, 2$. Note that, for $i>0$, the quads involved are $q_i, q_i', q_i'', q_i$, which all have the same support, while for $i=0$, the quads are $q_1, q_1', q_1'', q_1'$, the first two either have the same support or are related by an exchange of fixed points, while the support does not change afterwards.
\end{proof}

%(ac)(bd) --> (ad)(bc) --> (ab)(cd).

\begin{prop} \label{prop-graphs-An}
Let $\CC$ be a conjugacy class in $A_n$. Suppose that the elements of $\CC$ are products of $k$ quads with disjoint support (or equivalently that the support of an element of $\CC$ is a set of size $4k$). Then: \begin{enumerate}
\item if $n=4k$, the graph $\Gamma(\CC)$ is connected, 
\item if $n \ge 4k+2$, the graph $\Gamma(\CC)$ is connected, 
\item if $n= 4k+1$, the graph $\Gamma(\CC)$ has $n$ connected components. One of these is $\CC \cap A_{n-1}$.
\end{enumerate}
\end{prop}

\begin{proof}
(1) Let $s \in \CC$. We say that $s$ {\em involves} the transposition $\tau$ if, when we write $s$ as the product of distinct transpositions of disjoint support, $\tau$ occurs in the product. The element $s$ involves $(1a)$ for some $a$; if $a \ne 2$, then $s$ also involves $(2b)$ where $b \ne 1$. In this case $s=(1a)(2b) q_2 \cdots q_k$ where $q_i$ is a quad for $i>1$. By the last lemma, we have a path from $s$ to  $(12)(ab) q_2 \cdots q_k$.

Thus we may in fact suppose that $s$ involves $(12)$. Continuing, we can force the presence of $(34)$ and then $(56)$ and so on, until we have reached $(12)(34) \cdots (4k-1, 4k)$, following a path starting from $s$.

(2) Let $s \in \CC$, and let the support of $s$ be $\{ a_1, a_2, \ldots, a_{4k} \}$ with $a_i < a_{i+1}$. The argument of (1) applies, and gives a path between $s$ and $t= (a_1 a_2) (a_3 a_4)\ldots (a_{4k-1} a_{4k})$. Let $F$ denote the set of fixed points of $t$ and $A= F \cup \{ a_1, a_2, a_3, a_4 \}$. Then $A$ is left stable by the action of $t$, and certainly $\{ 1, 2, 3, 4 \}\subset A$; also note $|A| \ge 6$ by the assumption on $n$. If $q=(a_1a_2)(a_3 a_4)$ is the quad induced on $A$ by $t$, then the computations of Example \ref{ex-compute-quads} show that there is a sequence of moves from $q$ to $(12)(34)$; that is, there is a sequence of quads in $A$, say $q_1=q, q_2, \ldots, q_\ell=(12)(34)$, such that the last lemma applies between $q_i u$ and $q_{i+1} u$, where $u= (a_5 a_6)\ldots (a_{4k-1} a_{4k})$, for each index $i$. This shows the existence of a path between $t=q_1 u$ and $q_\ell u = (12)(34) (a_5 a_6) \cdots (a_{4k-1} a_{4k})$. Iterating, we have a path leading to $(12) (34) \cdots (4k-1, 4k)$.

(3) Now suppose $n= 4k+1$, so that if $s \in \CC$, then $s$ has just one fixed point $a$ with $1 \le a \le n$. Any permutation commuting with $s$ must leave $a$ fixed. Suppose $a=n$ to start with. Then the connected component of $\Gamma(\CC)$ containing $s$ only comprises involutions of $A_{n-1}$, and by (1) this connected component is in fact $\CC \cap A_{n-1}$. Similarly, each other choice of $a$ gives a connected component.
\end{proof}

Having worked out these connected components, we return to the proof of Theorem \ref{thm-binary-actions-An}, with the notation introduced right after its statement.

So suppose that $H$ has even order, and let $\CC$ be a conjugacy class of involutions of maximal 2-fixity. By Corollary \ref{coro-connected-comp}, we know that $H$ contains a component group of $\Gamma(\CC)$. Using the fact that we are assuming $n\geq 6$ and that $A_n$ is simple for $n\geq 5$,  Proposition \ref{prop-graphs-An} implies that the component groups must be all of $A_n$ when $n=4k$ or $n\ge 4k+2$, so that $H=G$ in this case. 

Likewise, when $n=4k+1$ the same proposition guarantees that $A_{n-1} \subset H$. The natural action of $A_n$ on $\{ 1, \ldots, n \}$ is obviously primitive, so $A_{n-1}$ is maximal in $A_n$, and thus we have either $H=A_{n}$ or $H= A_{n-1}$. In the latter case however, the action of $G$ on the cosets of $H$ is nothing but the natural action itself, which is not binary (Example \ref{ex-natural-alternating-action}). This contradiction shows that $H=G$ in all cases when $H$ has even order.

\subsection{When \texorpdfstring{$H$ is not a $3$-group}{H is not a 3 group}}

We continue the proof of Theorem \ref{thm-binary-actions-An}, now assuming that the order of $H$ is odd, and we seek a contradiction. First, we assume the existence of a prime number $p > 3$ which divides the order of $H$ (in other words, we exclude the case when $H$ is a $3$-group).

\begin{lem}
The group $H$ contains $p$-cycles.
\end{lem}

\begin{proof}
Let $\CC$ be a conjugacy class of $p$-elements of $G$ of maximal $p$-fixity, in the action of $G$ on the cosets of $H$. Certainly we can find $g \in \CC \cap H$. 

There is nothing to prove if the elements of $\CC$ are $p$-cycles, so assume that $g = c_1 c_2 \cdots c_k$ where $k>1$ and each $c_i$ is a $p$-cycle. Note that this means that $\CC$ contains all permutations which can be written as the product of $k$ disjoing $p$-cycles. Let $x\in \{2,3,\dots, p-2\}$ (this is possible as $p>3$). Now introduce 
\[ h = c_1^x c_2^{-1}\cdots c_k^{-1} \in \CC \, . \]
We have $[g,h]=1$ and $gh^{-1} = c_1^{1-x} c_2^2 \cdots c_k^2 \in \CC$, and we see that $g$ and $h$ are connected by an edge in $\Gamma(\CC)$. By Corollary \ref{coro-connected-comp}, we know that $h \in H$. However, we have
\[  gh = c_1^{1+x} \in H    \]
and we are done. 
\end{proof}

It will be useful to recall a few general facts about $p$-cycles. It seems more convenient to provide a direct argument rather than refer to the literature, and the following proposition has benefited from a conversation on MathOverflow (to be more precise, a question asked by Robinson and an answer from Müller \cite{mo_rm}). The fourth point of the proposition is of a somewhat different nature.

\begin{prop} \label{prop-p-cycles}
\begin{enumerate}
\item Let $K$ be a subgroup of $A_p$, where $p$ is prime, which is generated by two $p$-cycles. If $K$ is not abelian, then the order of $K$ is even. 
\item Let $s,t \in A_n$ be two cycles with supports $S$ and $T$. Assume that $S\cap T$, $S \smallsetminus T$ and $T \smallsetminus S$ are all nonempty. Then the group generated by $s$ and $t$ has even order.
\item Let $K$ be a subgroup of $A_n$ of odd order, and let $s, t \in K$ be $p$-cycles, where $p$ is a prime. Then either the supports of $s$ and $t$ are disjoint, or $\langle s \rangle = \langle t \rangle$. In particular, $s$ and $t$ commute.
\item Let $p$ be an odd prime, and let $s$ and $t$ be two $p$-cycles. Then $\langle s \rangle$ and $\langle t \rangle$ are conjugate in $A_n$.
\end{enumerate}
\end{prop}

\begin{proof}
(1) First, we note that the action of $K$ on $\{ 1, \ldots,  p\}$  is obviously primitive. We deduce that $K$ is simple. Indeed, if $N$ is a nontrivial, normal subgroup of $K$, then it must be transitive (by primitivity), and in particular its order must be a multiple of $p$. Hence it contains a Sylow $p$-subgroup of $K$ (since the order of such a subgroup is not divisible by $p^2$), and it follows that $N$ contains all the $p$-cycles in $K$, hence $N=K$.
    
We could rely on the Feit-Thomson theorem to argue that, if the order of $K$ were to be odd, then $K$ would be solvable as well as simple, hence $K$ would have to be abelian. However, we can avoid using such a strong result and appeal to Burnside's theorem on transitive permutation groups of prime degree $p$: this asserts that $K$ must be solvable or $2$-transitive. Since $K$ is simple, if $K$ is not abelian, then we must have that $K$ is $2$-transitive, and thus its order is certainly even.

(2) Pick any $x \in S \smallsetminus T$; pick $y \in T \smallsetminus S$ such that $y^t \in S$; and finally, pick $z \in S \cap T$ such that $z^t \in T \smallsetminus S$.
    
There is an integer $i$ such that $x^{s^i} = z$, and of course $y^{s^i} = y$. Apply $t$ and obtain that $x^{s^it} \in T \smallsetminus S$ while $y^{s^it} \in S$. Then pick an integer $j$ such that $y^{s^i t s^j} = x$, and observe that $x^{s^i t s^j} = x^{s^i t} \in T \smallsetminus S$. Finally, pick an integer $k$ so that $x^{s^its^jt^k} = y$, and note that $y^{s^i t s^j t^k} = x^{t^k} = x$.

The permutation $g = s^i t s^j t^k \in \langle s, t \rangle$ thus exchanges $x$ and $y$, so its order must be even. 

(3) If the supports of $s$ and $t$ are not disjoint, then they must be equal, by (2). In this case, by (1), $s$ and $t$ must commute. It is then elementary that $\langle s \rangle = \langle t \rangle$.

(4) First we show that, for any $p$-cycle $s$, there exists a permutation $\tau$ of signature $-1$ such that $s^\tau$ is a power of $s$. For this, it is notationally more convenient to consider $s = (0, 1, \ldots, p-1)$ in the symmetric group of the set $\{ 0, 1, \ldots, p-1 \}$. For any integer $k$, we have 
\[ s^k = (0, k, 2k, \ldots, (p-1)k) \]
where the entries are understood modulo $p$. Now pick for $k$ a generator of the cyclic group $\left( \z/p\z\right)^\times$, and let $\tau$ be the permutation defined by $i^\tau = ki$ (modulo $p$). Then $s^\tau = s^k$, but $\tau$ is just the cycle $(1, k, k^2, \ldots, k^{p-2})$ of length $p-1$, so the signature of $\tau$ is $-1$, as claimed.

We turn to the proof of statement (4) itself. Certainly there exists $\sigma \in S_n$ with $s^\sigma = t$. There is nothing to prove if $\sigma \in A_n$.  If the signature of $\sigma$ is $-1$, it suffices to pick a permutation $\tau$ with signature $-1$ which satisfies $t^\tau = t^k$ for some $k$ (such a permutation exists by the above paragraph), so that $\sigma \tau \in A_n$ and $s^{\sigma \tau} = t^k$. Thus $\langle s \rangle ^{\sigma \tau} = \langle t \rangle$.
\end{proof}

We return to the proof of Theorem \ref{thm-binary-actions-An} and recall that $p>3$ is a prime dividing $|H|$. As the order of $H$ is assumed to be odd, we deduce from (3) of the proposition the existence of $p$-cycles $c_1, \ldots, c_s$ with disjoint supports such that any $p$-cycle in $H$ is a power of some $c_i$. The subgroup $E = \langle c_1, \ldots, c_s \rangle \subset H$ is elementary abelian and normal.

We shall use Lemma \ref{lem-criterion-2-not-implies-3} to prove that the action of $G$ on the cosets of $H$ is not binary, which is the required contradiction.

We may as well assume that $c_1 = (1, 2, \ldots, p)$, and we put $h_1 = c_1$. Consider $\sigma = (p-2, p-1, p) \in A_n$ and 
\[ h_2 = h_1^\sigma = (1, 2, \ldots, p-3, p-1, p, p-2) \, . \]
It is easy to see that 
\[ h_1 h_2 = (2, 4, 6, \ldots, p-3, 1, 3, 5, \ldots, p-4, p-1, p-2, p ) \, . \]
%

% Moreover we have $h_1 h_2 = h_1^\tau$ where 
% %
% \[ \tau = \left(\begin{array}{rrrrrrrrrrrrr}
%  1 & 2 & 3 & \cdots & (p-3)/2 & \cdots & \cdots & & & \cdots & p-2 & p-1 & p  \\ 
%  2 & 4 & 6 & \ldots & p-3& 1 & 3 & 5 & \ldots & p-4 & p-1 & p-2 & p  
% \end{array}\right) \, .\]
% %
% We claim that $\tau \in A_n$. For this, we count the inversions of $\tau$, that is the pairs $i < j$ with $\tau(i) > \tau(j)$. We have such an inversion when $\tau(i)$ is an even number between $2$ and $p-3$ and $\tau(j)$ is an odd number between $1$ and $\tau(i) - 1$, giving $\frac{\tau(i)}{2}$ inversions for each such value of $\tau(i)$. In total this produces
% %
% \[  \sum_{k=1}^{\frac{p-3}{2}} k = \frac{1}{2} \frac{(p-3)}{2} \frac{(p-1)}{2} \]
% %
% inversions, and we should add the inversion with $i=p-2$, $j=p-1$. 
%
We put $h_3 = (h_1h_2)^{-1}$ so that $h_1h_2h_3= 1$, and $h_3$ is a $p$-cycle. By (4) of Proposition \ref{prop-p-cycles}, we may choose $\tau \in A_n$ so that $h_1^\tau$ is a power of $h_3$; moreover we can choose $\tau$ with support in $\{ 1, 2, \ldots, p \}$ (by applying the proposition to $A_p$ rather than $A_n$). Thus if we put $H_1 = H$, $H_2 = H^\sigma$ and $H_3 = H^\tau$, then $h_i \in H_i$ for $i= 1, 2, 3$, and $H_i$ is the stabilizer of some coset of $H$ in $G$.

Note that, since $h_1$ and $h_2$ are clearly not powers of each other, the same is true of $h_1h_2$ and so $H_1$, $H_2$ and $H_3$ are distinct.

We now apply (2) from Lemma \ref{lem-criterion-2-not-implies-3}. As we assume that we are dealing with a binary action, we conclude that it must be possible to pick $h_2' \in H_1 \cap H_2$ and $h_3' \in H_3$ such that $h_1 h_2' h_3' = 1$. 

Just as $H_1 = H$ has the normal subgroup $E_1 = E = \langle h_1, c_2, \ldots, c_s \rangle$, the subgroup $H_i$ has the normal subgroup $E_i = \langle h_i, c_2, \ldots, c_s \rangle$, for $i= 2, 3$ (this remark uses the observation that for $i > 1$, we have $c_i^\sigma = c_i^\tau = c_i$). It follows that $H_1 \cap H_2$ normalizes $E_1 \cap E_2 = \langle c_2, \ldots, c_s \rangle$, as well as $E_1$ and $E_2$ individually. If $\Lambda \subset \{ 1, 2, \ldots, n \}$ denotes the (disjoint) union of the supports of the elements $c_i$ for $1 \le i \le s$, we see that $H_1 \cap H_2$ preserves $\Lambda$ but also $\Lambda \smallsetminus \{ 1, \ldots, p \}$ -- the latter being the union of the supports of the $c_i$'s for $1 < i \le s$. As a result, $H_1 \cap H_2$ preserves $\Delta =  \{ 1, \ldots, p \}$.

In particular, $h_2'$ induces a permutation of $\Delta$, and the latter must normalize both $\langle h_1 \rangle$ and $\langle h_2 \rangle$. However, we have the following lemma.

\begin{lem} \label{lem-inter-normalizers}
Let $N_i$ be the normalizer of $\langle h_i \rangle$ in $S_p$, for $i=1, 2$. Then $N_1 \cap N_2$ is trivial for $p>5$, and has order $4$ for $p=5$.
\end{lem}

We postpone the proof of the lemma and explain how we derive a contradiction. Indeed, as $H$, and so also $h_2'$, has odd order, the lemma implies that $h_2'$ is the identity on $\Delta$. By inspection of the relation $h_1 h_2' h_3' = 1$, we deduce that $h_3'$ also preserves $\Delta$ and indeed that $h_3' = h_1^{-1}$ on $\Delta$. Then $h_3'$ is in the setwise stabilizer of $\Delta$ in $H_3$ and acts as a $p$-cycle on $\Delta$; thus $h_3'$ acts on $\Delta$ as a power of $h_3$. Since $h_1^{-1}$ is clearly not a power of $h_3$ we obtain a contradiction.

%However, argueing as above we see that $h_3'$ must normalize $\langle h_1   \rangle$ and, more to the point, also $\langle h_3 \rangle$. This is impossible, as we would draw that $h_1$ normalizes the group generated by $h_3^{-1} = h_1 h_2$, when in fact 
%
%\[ (h_1 h_2)^{h_1} = (3, 5, \ldots, p-2, 2, 4, 6, \ldots, p-3, p-1, 1) \, . \]
%
%This cycle is clearly not a power of $h_1h_2$. \pierre{this cycle and $h_1h_2$ both have the subsequence $3,5$ in common, so if one were a power of the other they would be equal, so $h_1$ would commute with $h_1 h_2$ and so with $h_2$, but it does not. I don't think I want to write this. What do you think?}

This contradiction concludes the proof of Theorem \ref{thm-binary-actions-An} in the case under consideration, namely when $H$ has odd order but is not a $3$-group.

\begin{proof}[Proof of Lemma \ref{lem-inter-normalizers}] We give a proof for $p \ge 11$, as it allows for an easier exposition; the cases $p=5$ and $p=7$ can be worked out by direct computation (we have used a computer for this, for safety). Again it is more convenient to work in the symmetric group of $\z/p\z = \{ 0, 1, \ldots, p-1 \}$, and so we consider 
\[ h_1 = (0, 1, \ldots, p-1) \, , \quad  \sigma = (p-3, p-2, p-1) \, ,   \]
and 
\[ h_2 = h_1^\sigma = (0, 1, \ldots, p-4, p-2, p-1, p-3) \, . \]
We pick a permutation $g$ which normalizes both $\langle h_1 \rangle$ and $\langle h_2  \rangle$, and we wish to show that $g=1$. 

It is classical that the normalizer of $\langle h_1 \rangle$ consists of all permutations $g$ with $x^g = kx +t$ for $x \in \z/p\z$, where $k \in \left(\z/p\z\right)^\times$ and $t \in \z/p\z$. With this notation, we have $h_1^g = h_1^k$.

Now let $\ell \in \left(\z/p\z\right)^\times$ be such that $h_2^g = h_2^\ell$. First we show that $k = \ell$. For this, we put $s = h_2^g$ and study the values of the form $x^s - x$ for $0 \le x < p$. More precisely, put 
\[ S= \{ x : x^s - x = k ~\textnormal{mod}~ p \} \, , \quad T= \{ x : x^s - x = \ell ~\textnormal{mod}~p \} \, ;\]
we shall show that $S \cap T \ne \emptyset$. On the one hand, for $x = y^g$ with $0 \le y < p-4$, we have 
\[ x^s - x = (y^{h_2})^g - y^g = (y+1)^g - y^g = k \, . \]
And so the size of $S$ is at least $p-4$. On the other hand, we have $s = h_2^\ell$. It is obvious that, if $0 \le x \le p-4$ and $0 \le x+\ell \le p-4$, then $x^{h_2^\ell} = x + \ell$; this implies that, for the same values of $x$ and for $x+\ell > p-1$, we have $x^{h_2^\ell} = x + \ell$ modulo $p$. We have found that $x^s - x = \ell$ modulo $p$ for all but 6 values of $x$, namely, the exceptions could occur when $x \in \{ p-3, p-2, p-1 \}$ or when $x + \ell  \in \{ p-3, p-2, p-1 \}$. In other words the size of $T$ is no less than $p-6$. We see that $S$ must intersect $T$ nontrivially, lest we should conclude that $p \le 10$. We have proved that $k=\ell$ modulo $p$.

Next we write $(h_1^k)^{\sigma^g} =(h_1^g)^{\sigma^g} = (h_1^\sigma)^g = h_2^g =  h_2^k = (h_1^\sigma)^k = (h_1^k)^\sigma$. In particular from $(h_1^k)^{\sigma^g} = (h_1^k)^\sigma$ we see that $\sigma^g \sigma^{-1}$ centralizes the $p$-cycle $h_1^k$, and so also $h_1$ itself. However, we have recalled the description of the normalizer of $\langle h_1 \rangle$, from which it follows that the centralizer of $h_1$ is $\langle h_1 \rangle$. As $\sigma$ is a 3-cycle, we see that $\sigma^g \sigma^{-1}$ moves at most 6 points, and so it cannot be a $p$-cycle with $p \ge 7$. The only element in $\langle h_1 \rangle$ which is not a $p$-cycle is the identity, so we conclude that $\sigma^g = \sigma$.

Thus $g$ preserves the support of $\sigma$, and the permutation $g_0$ induced by $g$ on this support is a power of $\sigma$. Consider the possibility $g_0 = \sigma$. This leads to the equations 
\[ \left\{ \begin{array}{ll}
 k(p-3) + t & = p-2  \\ 
 k(p-2) + t & = p-1 \\
 k(p-1) + t & = p-3 
 \end{array}\right.
 \]
which have no solutions $k, t \in \z/p\z$ unless $p=3$. Similarly, $g_0 = \sigma^{-1}$ leads to a contradiction. There remains only $g_0 = 1$. Writing the associated system of equations shows readily that $k=1$ and $t=0$, so that $g= 1$.
\end{proof}

\subsection{When \texorpdfstring{$H$ is a $3$-group}{H is a 3 group}}

We finish the proof of Theorem \ref{thm-binary-actions-An} in the remaining case, that is, when $G=A_n$ acts on the cosets of a subgroup $H$ which is a non-trivial $3$-group. We asssume that the action is binary and look for a contradiction.

\begin{lem}\label{l: n9 3c}
For $n \ge 9$, the subgroup $H$ must contain $3$-cycles.
\end{lem}

\begin{proof}
Let $\CC$ denote a conjugacy class of $3$-elements of $G$ of maximal 3-fixity and note that $H\cap\CC\neq\emptyset$. There is nothing to prove if $\CC$ contains $3$-cycles, so assume that $\CC$ is the conjugacy class of $c_1 \cdots c_k$ where each $c_i$ is a $3$-cycle and $k > 1$. 

It is easy to check that, when $\CC$ is the conjugacy class of $(123)(456)(789)$ in $A_9$, then the graph $\Gamma(\CC)$ is connected. In particular, it has a single component group which is the whole of $A_9$. It follows from the next lemma that, for $k \ge 3$, the component groups of $\Gamma(\CC)$ have even order. This is absurd, as $H$ contains such a group by Corollary \ref{coro-connected-comp}.

Thus we must only consider the possibility $k=2$. In $\Gamma(\CC)$ we see that, if $n\geq 9$, then $h_1=(123)(456)$ is connected to $h_2=(123)(789)$ and $h_2$ is, in turn, connected to $h_3=(123)(465)$. Thus $h_1$ and $h_3$ are in the same connected component of $\Gamma(\CC)$; moreover, the product of these two elements is $(132)$. By Corollary \ref{coro-connected-comp}, we see that $H$ contains $3$-cycles.
\end{proof}

% begin

% And now, a lemma that is useful when you fix $p$ and try to prove things for greater and greater values of $k$. We'll need it for $p=3$, I guess. Unfortunately we need some heavier notation: let $C_{k}^n$ denote the conjugacy class, within $A_n$, of a product of $k$ different $p$-cycles with disjoint supports. (For $k \ge 2$ all such products are conjugate in $A_n$.) Whenever $C$ is a conjugacy class of $p$-elements, for lack of a better notation I'll write $[C]$ for the group generated by a connected component of $\Gamma(C)$ (thus $[C]$ is defined up to conjugacy).

In the course of the proof, we have relied on the following lemma (which we only need for $p=3$ of course):

\begin{lem}
Let $p$ be an odd prime, and let $k \ge 2$. Let $\CC$ denote the conjugacy class, in the group $A_{pk}$, containing all products of $k$ different $p$-cycles with disjoint supports. Suppose that the component groups of $\Gamma(\CC)$ have even order. Then for all $\ell \ge k$ and all $n \ge p\ell$, the component groups of $\Gamma(\CC')$ have even order, where $\CC'$ is the conjugacy class, in the group $A_n$, of a product of $\ell$ different $p$-cycles with disjoint supports. 
\end{lem}

\begin{proof}
By induction it is enough to prove this for $\ell = k+1$ and any $n \ge (k+1)p$.  Let $c$ be any $p$-cycle with support disjoint from $\{ 1, 2, \ldots, pk \}$.

If $g, h \in \CC$ are joined by an edge in $\Gamma(\CC)$, we put $g' = gc$ and $h' = hc^{-1}$. Then $g', h' \in \CC'$, $[g', h']= 1$, and $g'(h')^{-1} = (gh^{-1}) c^2 \in \CC'$ (since $p$ is odd). Thus $g'$ and $h'$ are joined by an edge in $\Gamma(\CC')$. 

It follows readily that, if $g, h$ are in the same component of $\Gamma(\CC)$, then $gc$ and $hc^{\pm 1}$ are in the same component of $\Gamma(\CC')$, for some sign $\pm 1$. The lemma is now clear.
\end{proof}

Now we complete the proof of Theorem~\ref{thm-binary-actions-An}. Assume, first, that $H$ contains a $3$-cycle.
We deduce again from (3) of Proposition \ref{prop-p-cycles} (which holds for $p=3$ as well) the existence of $3$-cycles $c_1, \ldots, c_s$ with disjoint supports such that any $3$-cycle in $H$ is a power of some $c_i$. The subgroup $E = \langle c_1, \ldots, c_s \rangle \subset H$ is elementary abelian and normal. We will now assume that $s \ge 2$, leaving the case $s=1$, which is similar but easier, to the reader.

We assume that $c_1= (123)$ and $c_2= (456)$. Put $\sigma = (23)(14)$, $\tau = (234)$ and define $h_1= c_1$, $h_2= c_1^\sigma = (432)$ and $h_3 = h_1^\tau = (134)$. We have $h_1h_2 h_3 = 1$ and $h_i \in H_i$ for $i=1,2,3$, where $H_1= H$, $H_2 = H^\sigma$ and $H_3 = H^\tau$.

As above, we now apply (2) from Lemma \ref{lem-criterion-2-not-implies-3}. We conclude that it must be possible to pick $h_2' \in H_1 \cap H_2$ and $h_3' \in H_3$ such that $h_1 h_2' h_3' = 1$. 

We note that $c_2^\sigma = (156)$ and $c_2^\tau = (256)$, while $c_i^\sigma = c_i^\tau = c_i$ for $i>2$. Following the argument for $p>3$, we arrive at the conclusion that $h_2'$ stabilizes $\Delta = \{ 1,2,3,4,5,6 \}$, and that it must normalize both $\langle h_1, c_2 \rangle$ and $\langle h_2, c_2^\sigma \rangle$. A direct computation shows that the intersection of the normalizers in $A_6$ of these subgroups has order $4$; as $h_2'$ has odd order, we conclude that $h_2'$ induces the identity of $\Delta$. A similar reasoning and computation with $h_3'$ shows that this permutation is also the identity on $\Delta$. 

This is absurd, however, as we see from the identity $h_1 h_2' h_3'=1$, since $h_1$ is not the identity on $\Delta$. This completes the proof of Theorem~\ref{thm-binary-actions-An} in the case where $H$ is a 3-group containing a 3-cycle. Lemma~\ref{l: n9 3c} then yields the proof provided $n\geq 9$. 

We are left with the case when $n\in\{6,7,8\}$. If $n=6$ and $H$ contains a $3$-cycle, then we are already done. If $n=6$ and $H$ does not contain a $3$-cycle, then $H=\langle h\rangle$ where $h$ is the product of two 3-cycles. But now the action of $G=A_6$ on the set of cosets of $H$ is permutation isomorphic (via an exceptional outer automorphism) to the action of $G=A_6$ on the set of cosets of a subgroup $H'$ generated by a $3$-cycle, so the proof is complete in this case too. 

Finally if $n\in\{7,8\}$, then there exists $M\cong A_6$ such that $H<M<G$. We know that the action of $M$ on the set of cosets of $H$ in $M$ is not binary and so the same is true for the action of $G$ by \cite[Lemma~1.7.2]{gls_binary}.
This final contradiction concludes the proof of Theorem \ref{thm-binary-actions-An} when $n \ge 9$.

\bibliography{myrefs}
\bibliographystyle{amsalpha}

\end{document}